\documentclass[letterpaper, conference]{ieeeconf}  

\IEEEoverridecommandlockouts                              
\usepackage{graphics} 

\title{\LARGE \bf
Approximate observability and back and forth observer\\
of a PDE model of crystallisation process
}

\author{Lucas Brivadis and Ludovic Sacchelli
\thanks{This research was partially funded by the French Grant ANR ODISSE (ANR-19-CE48-0004-01)}
\thanks{L. Brivadis and L. Sacchelli are with Univ. Lyon, Universit\'e Claude Bernard Lyon 1, CNRS, LAGEPP UMR 5007, 43 bd du 11 novembre 1918, F-69100 Villeurbanne, France
        {\tt\small lucas.brivadis@univ-lyon1.fr, ludovic.sacchelli@univ-lyon1.fr}}%
}

\usepackage{amsmath, amssymb, amsfonts}

\usepackage{cite}

\newtheorem{theorem}{Theorem}[section]

\newtheorem{remark}[theorem]{Remark}
\newtheorem{definition}[theorem]{Definition}

\newcommand{\F}{\mathcal{F}}
\newcommand{\K}{\mathcal{K}}
\newcommand{\diff}{\mathrm{d}}

\newcommand{\regis}{\textsuperscript{\mbox{\scriptsize{\textregistered}}} }

\newcommand{\tmax}{T}
\newcommand{\dtc}{\psi}
\newcommand{\rmin}[1]{r^{\min}_{#1}}
\newcommand{\rmax}[1]{r^{\max}_{#1}}
\newcommand{\rinf}[1]{r^{\inf}_{#1}}
\newcommand{\rsup}[1]{r^{\sup}_{#1}}
\newcommand{\vit}{G}

\newcommand{\noy}{k}
\newcommand{\ee}{\eta}
\newcommand{\opc}{\mathcal{K}}
\newcommand{\lmax}{\ell^{\max}}
\newcommand{\R}{\mathbb{R}}
\newcommand{\etath}{\hat{\psi}}
\DeclareMathOperator{\sign}{sign}

\newcommand{\fonction}[5]{
\begin{align*}
\displaystyle
\begin{array}{lrcl}
#1: & #2 & \longrightarrow & #3 \\
    & #4 & \longmapsto & #5
\end{array}
\end{align*}}

\newcommand{\cv}{\to}
\newcommand{\cvl}[1]{\underset{#1}{\longrightarrow}}

\usepackage[svgnames]{xcolor}

\begin{document}

\maketitle
\thispagestyle{empty}
\pagestyle{empty}

\begin{abstract}

In this paper, we are interested in the estimation of Particle Size Distributions (PSDs) during
a batch crystallization process
in which particles of two different shapes coexist and evolve simultaneously.
The PSDs are estimated thanks to a measurement of an apparent Chord Length Distribution (CLD), a measure that we model for crystals of spheroidal shape.
Our main result is to prove the approximate observability of the infinite-dimensional system in any positive time.
Under this observability condition, we are able to apply a Back and Forth Nudging (BFN) algorithm to reconstruct the PSD.
\end{abstract}

\section{Introduction}

During a batch crystallization process, a critical issue is to monitor the Particle Size Distribution (PSD), which may affect the chemical-physical properties of the product.
In particular, multiple types of crystals may be evolving simultaneously in the reactor, with stable and meta-stable crystal formations. 
In that case, estimating the PSD associated to each shape is an important task, but difficult to realize in practice.
Modern Process Analytical Technologies (PATs) offer a wide variety of approaches to extract PSD information from measurements, such as image processing  based methods \cite{Benoit, GAO} for instance.
The use dynamical observers is a popular approach to the issue \cite{vissers2012model, porru2017monitoring, mesbah2011comparison, Uccheddu, lebaz, gruy:hal-01637703, BRIVADIS202017}.
A particular technique, on which we focus in this article, 
is to use PATs giving access to the Chord Length Distribution (CLD) (such as the Focused Beam Reflectance Measurement or the BlazeMetrics\regis technologies), and to reconstruct the PSD from the knowledge of the CLD \cite{worlitschek2005restoration, liu1998relationship, pandit2016chord, AGIMELEN2015629}.
When scanning across crystals, these sensors actually measure chords on the projection of the crystal on the plane that is orthogonal to the probe.
Hence, the measured CLD highly depends on the shapes that the crystals in the reactor may take.
And in many crystallization processes, several shapes can coexist due to polymorphism. Since the CLD sums up the contribution of each shape in one measurement, recovering the PSD associated to each shape only from the common CLD is a major challenge not yet tackled by the existing literature.

In the previous work \cite{brivadis:hal-03053999}, we proposed:
\textit{(i)} a model of the PSD-to-CLD relation for spheroid particles;
\textit{(ii)} a direct inversion method to instantly recover the PSD from the CLD when all particles have the same shape;
\textit{(iii)} a back and forth observer to reconstruct the PSD of several shapes from the 
knowledge of their common CLD on a finite time interval and an evolution model of the process.
We were able to prove the convergence of the algorithm only in one case: when crystals have only two possible shapes (spheres and elongated spheroids of fixed eccentricity), each having a positive growth rate independent of the size (McCabe hypothesis).
Among the differences with the previous result, let us highlight the two main improvements of this paper:
\textit{(a)}  we consider size-dependent growth rates;
\textit{(b)} we consider almost all possible combinations of two spheroid shapes with different eccentricities.
Then, we perform an observability analysis of the infinite-dimensional system, which is the main result of the paper.





\section{Evolution model and CLD}

A spheroid is a surface of revolution, obtained by rotating an ellipse along one of its axes of symmetry.
In particular, spheres are spheroids.
A spheroid is fully described by two scalar parameters: a radius $r$, characterizing its size and being the radius that is orthogonal to its rotation's axis, and an eccentricity $\eta$, characterizing its shape and being the ratio between the radius along its rotation's axis and $r$.

Consider a batch crystallisation process during which crystals of only two shapes appear: spheroids of eccentricities $\eta_1$ and $\eta_2$.
Let $\psi_1$ and $\psi_2$ be their corresponding PSDs. At any time $t\in[0, \tmax]$ during the process,
$\int_{r_1}^{r_2}\psi_i(t, r)\diff r$
is the number of crystals per unit of volume at time $t$
having the shape $\eta_i$ and a radius $r$ between $r_1$ and $r_2$.

Let $\rmin{i}$ be the minimal size at which crystals of shape $\eta_i$ appear
and $\rmax{i}$ be a maximal radius that no crystals can reach during the process.
At time $t=0$, assume that seed particles with PSD $\psi_{i, 0}$ for each shape $\eta_i$ lie in the reactor.
Denote by $G_i(t, r)$ the grow rate of crystals of shape $\eta_i$ and size $r$ at time $t$.
The usual (see, \emph{e.g.}, \cite{Mullin, Mersmann}) population balance equation leads to
\begin{equation}
    \frac{\partial \dtc_i}{\partial t}(t, r) + \vit_i(t, r) \frac{\partial \dtc_i}{\partial r}(t, r) = 0
\end{equation}
Assume that $G_i$ is $C^1$.
Note that the growth rate may be positive or negative.
The boundary conditions are given by
\begin{align}
    &\dtc_i(t, \rmin{i}) = u_i(t)
    \label{eq:rmin}
    \\
    &\dtc_i(t, \rmax{i}) = 0
    \label{eq:rmax}
\end{align}
where $u_i(t)$ denotes the appearance of particles of size $\rinf{i}$ and shape $\eta_i$ at time $t$.
Since $u_i$ is supposed to be unknown, it is part of the data to be estimated, with $\psi_{i, 0}$.
Note that the boundary conditions impose a relation between $u_i$ and $\psi_{i, 0}$ when $G_i(t, \rmin{i})<0$.
Set 
\begin{align}
&\begin{aligned}
\rinf{i} = \min\Big\{&\rmin{i},
\rmin{i} - \max_{\tau\in[0, \tmax]} \int_0^\tau G_i(t, \rmin{i})\diff t\Big\}
\end{aligned}
\\
&\begin{aligned}
\rsup{i} = \max\Big\{&\rmax{i},
\rmax{i} - \min_{\tau\in[0, \tmax]} \int_0^\tau  G_i(t, \rmax{i})\diff t\Big\}
\end{aligned}
\end{align}
In order to ensure the well-posedness of the evolution equation,
let us define $\psi_i(t, r)$ and $G_i(t, r)$ for $t\in[0, \tmax]$ and $r\in[\rinf{i}, \rsup{i}]\setminus[\rmin{i}, \rmax{i}]$ by
\begin{align}
&\left\{\begin{aligned}
&\dtc_i(t, r) = u_i(t+\tau),\\
&G_i(t, r) = G_i(t, \rmin{i}),
\end{aligned}\right.
&\forall r\in[\rinf{i}, \rmin{i}]
\label{eqtau}
\\
&\left\{\begin{aligned}
&\dtc_i(t, r) = 0,\\
&G_i(t, r) = G_i(t, \rmax{i}),
\end{aligned}\right.
&\forall r\in[\rmax{i}, \rsup{i}]
\end{align}
where $\tau$ in \eqref{eqtau} is such that
\begin{equation}
    \rmin{i} = r + \int_t^{t+\tau}G_i(s, \rmin{i})\diff s
\end{equation}
Roughly speaking, $\psi_i(t, r)$ for $r<\rmin{i}$ represents crystals that did not yet appear at time $t$, but will appear later at some time $t+\tau$.

Then the evolution of the crystallization process can be modeled as
\begin{equation}
\left\{\begin{aligned}
&\frac{\partial \dtc_i}{\partial t}(t, r) + \vit_i(t, r) \frac{\partial \dtc_i}{\partial r}(t, r) = 0\\
&\dtc_i(0, r) = \dtc_{i, 0}(r)
\end{aligned}
\right.
\label{syst}
\end{equation}
where $i\in\{1, 2\}$, $t\in[0, \tmax]$, $r\in[\rinf{i}, \rsup{i}]$
and
with the periodic boundary condition $\psi_i(t, \rinf{i}) = \psi_i(t, \rsup{i})$ since the boundary terms does not influence $\psi_i(t, r)$ for $r\in [\rmin{i}, \rmax{i}]$ and $t\leq \tmax$.
The new initial condition $\psi_{i, 0}(r)$ contains both the information on the seed particles (for $r\in[\rmin{i}, \rmax{i}]$ and on all the crystals that will appear during the process (for $r\in[\rinf{i}, \rmin{i}]$).

Note that, contrary to \cite{brivadis:hal-03053999},
we do not assume that $\rmin{1} = \rmin{2}$, nor that $G_i(t, r)$ is independent of $r$. We rather make the assumption that $G_i$ has separate variables, that is,
\begin{equation}
    G_i(t, r) = g_i f(t) h(r)
\end{equation}
for all $t\in[0, \tmax]$ and all $r\in[\rmin{i}, \rmax{i}]$,
where $g_i$ is a constant (either positive or negative, depending on whether crystals of shape $\eta_i$ are appearing or disappearing), and $f$ and $h$ do not depend on $i$.

\begin{remark}
In modelling the growth rate $G_i$, it can be linked with individual crystal volume growth.
For $r$ the radius of an individual crystal, the volume of an individual crystal is $V=\frac{4}{3}\pi \eta r^3$. As a consequence, $\frac{d V}{d t }=4\pi \eta r^2 G_i$. This leads to choices such as $h(r)=1/r^2$ for linear volume growth, or $h(r)=1$ for volume growth proportional to the crystal surface (which corresponds to McCabe hypothesis).

\end{remark}

Denote by $L^2(\rinf{i}, \rsup{i})$ the set of square integrable real-valued functions over $(\rinf{i}, \rsup{i})$, and $H^p(\rinf{i}, \rsup{i})$ the usual real-valued Sobolev spaces for $p\in\mathbb{N}$.

\begin{theorem}[Well-posedness]\label{th:wp}
Assume that $f\in C^0([0, \tmax]; \R)$ has a finite number of zeros
and $h$ is Lipschitz over $[\rmin{i}, \rmax{i}]$ and has constant sign.
Then for all $\psi_{i, 0}\in L^2(\rinf{i}, \rsup{i})$,
there exists a unique solution $\psi_i\in C^0([0, \tmax]; L^2(\rinf {i}, \rsup{i}))$ of the Cauchy problem \eqref{syst}.
\end{theorem}
\begin{proof}
Let $n$ be the number of zeros of $f$
and $([t_k, t_{k+1}])_{1\leq k\leq n}$ be intervals on which $f$ has constant sign, with $t_1 = 0$ and $t_n = \tmax$.
Over each interval $[t_k, t_{k+1}]$, introduce the time reparametrization $\tilde{t} = \int_{t_k}^t |f(s)|\diff s$.
Then $\psi$ is a solution of \eqref{syst} over $[t_k, t_{k+1}]$ if and only if $\tilde{\psi}(\tilde{t})= \psi(t)$ is a solution
\begin{equation}
\frac{\partial \tilde{\dtc}_i}{\partial \tilde{t}}(\tilde{t}, r) 
+ (\sign f) g_i h(r) \frac{\partial \tilde{\dtc}_i}{\partial r}(\tilde{t}, r) = 0
\end{equation}
According to \cite[Appendix 1]{bastin2016stability},
this linear hyperbolic system with periodic boundary conditions admits a unique solution in $C^0([t_k, \int_{t_k}^{t_{k+1}} |f(s)|\diff s]; L^2(\rinf{i}, \rsup{i}))$.
Reasoning by induction on each interval $[t_k, t_{k+1}]$, we find that there exists a unique $\psi_i\in C^0([0, \tmax]; L^2(\rinf{i}, \rsup{i}))$ solution of \eqref{syst}.
\end{proof}

\begin{remark}
If $\psi_{i, 0}$ and $f$ are more regular, then the corresponding solutions of \eqref{syst} are also more regular. In particular, if 
$\psi_{i, 0}\in H^2(\rinf{i}, \rsup{i})$ and
$f\in C^2([0, \tmax], \R)$,
then
$\psi_i\in C^0([0, \tmax]; H^2(\rinf {i}, \rsup{i})) \cap C^1([0, \tmax]; H^1(\rinf {i}, \rsup{i}))\cap C^2([0, \tmax]; L^2(\rinf {i}, \rsup{i}))$.
This remark will be used in Theorem~\ref{th:main}.
\end{remark}

Now, let us recall the model of the accessible measurement, the CLD, denoted by $q$, given in \cite{brivadis:hal-03053999}.
For any $t\in[0, \tmax]$, $\int_{\ell_1}^{\ell_2} q(t, \ell)\diff \ell$ is the number of chords measured by the sensor at time $t$ with length $\ell$ between $\ell_1$ and $\ell_2$.
The cumulative CLD is given by $Q(t, \ell) = \int_0^\ell q(t, l)\diff l$.
We model the PSD-to-CLD relation by
\begin{multline}\label{eqQ}
    Q(t, \ell)
    = \int_{\rmin{1}}^{\rmax{1}} k_1(\ell, r)\psi_1(t, r)\diff r\\
    + \int_{\rmin{2}}^{\rmax{2}} k_2(\ell, r)\psi_2(t, r)\diff r
\end{multline}
where $k_i$ is the  kernel for the PSD-to-CLD of each crystal shape $i$.
As developed \cite{brivadis:hal-03053999}, the apparent shape of a crystal with respect to the sensor is that of an ellipse (by projection onto a plane). In that way, for a given ellipse in the plane, the probability that the measured chord-length is less than  $\ell$ is  $1- \sqrt{1-\alpha\ell^2/(4r^2)}$, with coefficient $\alpha>0$ depending on orientation and eccentricity of the ellipse (with maximum possible chord length $2r/\sqrt{\alpha}$). The apparent ellipse is linked to the shape of the crystal and the crystal's random orientation in the suspension, which follows a uniform distribution on the sphere given in spherical coordinates by the probability measure $\frac{\sin \theta}{4\pi} \diff \phi\diff \theta$ for $(\phi,\theta)\in[0,2\pi]\times [0,\pi]$.
The kernel $k_i$
is obtained by {total expectation over possible orientations:}
\begin{equation}\label{eq:ker}
\noy_i(\ell,r)
=
1-
\int_{\phi=0}^{2\pi}\int_{\theta=0}^\pi
\sqrt{1-\dfrac{\ell^2}{4r^2}\alpha_{\ee_i}(\phi,\theta)}\frac{\sin \theta}{4\pi} \diff \theta\diff \phi,
\end{equation}
with
\begin{equation}
\alpha_{\ee_i}(\phi,\theta)=\frac{\cos ^2\phi}{\cos^2\theta+\ee_i^2\sin^2 \theta }+\sin ^2\phi
\end{equation}
and the convention that $\sqrt{x}=0$ for $x<0$.
Expression \eqref{eqQ} comes from the law of total expectation, while kernels $k_i$ can be determined by a probabilistic analysis of the two sources of hazards in the measure of a chord on a spheroid crystal: the random orientation of the spheroid with respect to the probe, and the random chord measured by the sensor on the projection of the spheroid onto the plane that is orthogonal to the probe’s laser beam.
Note that in the particular case of spherical crystals (\emph{i.e.} $\eta=1$), expression \eqref{eq:ker} is simpler since $\alpha_1(\phi, \theta)=1$.

For a given shape $\eta$, the length of the largest chord possibly measured by the sensor on a crystal of size $r$ is $\lmax = 2r\max\{1, \eta\}$, since the direction of the largest diameter of a spheroid depends on whether $\eta>1$ or not.
Set $\lmax = 2\max\{\rmax{1}\max\{1, \eta_1\}, \rmax{2}\max\{1, \eta_2\}\}$.
Let $X_i = L^2(\rinf{i}, \rsup{i})$ be the function spaces of PSDs and $Y = L^2(0, \lmax)$ the function space of CLD.
Define the operator $\mathcal{K}:X_1\times X_2\to Y$ that maps PSDs to their corresponding CLD:
\begin{equation}
    \opc(\psi_1, \psi_2) = \left(Q:\ell\mapsto
\sum_{i=1}^2 \int_{\rmin{i}}^{\rmax{i}} \noy_i(\ell, r)\psi_i(r)\diff r\right).
\end{equation}
The estimation problem that we aim to solve in this paper is the following:
``From the knowledge of $Q(t) = \opc(\psi_1(t), \psi_2(t))$ over $[0, \tmax]$, where $(\psi_1, \psi_2)$ is a solution of \eqref{syst},
estimate  $(\psi_{1, 0}, \psi_{2, 0})$.''

\section{Observability analysis}

First, we need to determine if the CLD $Q$ contains enough information to reconstruct the two PSDs $\psi_1$ and $\psi_2$.
In other words, we investigate the observability of the PDE \eqref{syst} with measured output $Q$.
Several observability notions exist on infinite-dimensional systems.

\begin{definition}[Observability]
Let $(\psi_1, \psi_2)$ be a solution of \eqref{syst} and $Q$ be the corresponding CLD.
Let $W(\tmax) = \int_0^{\tmax} \|Q(t)\|_Y^2 \diff t$. System \eqref{syst} is said to be
\begin{itemize}
    \item  exactly observable if, for some $\kappa>0$,
    $W(\tmax) \geq \kappa(\|\psi_{1, 0}\|_{X_1}^2+\|\psi_{2, 0}\|_{X_2}^2)$ for all $\psi_{i, 0}\in X_i$;
    \item approximately observable if
    $W(\tmax) >0 $ for all $(\psi_{1, 0}, \psi_{2, 0})\neq(0,0)$.
\end{itemize}
\end{definition}
These two notions are widely discussed in \cite{TW} for example.
Clearly, exact observability implies approximate observability, and they are equivalent on finite dimensional systems.
Unfortunately, the function $k_i(\ell, r)$ being bounded, the system is not exactly observable according to \cite[Proposition 6.3]{brivadis:hal-02529820}.
Therefore, we focus on approximate observability, as in \cite{Celle-etal.1989, xu1995observer, haine2014recovering}, and more recently in \cite{brivadis:hal-02529820}.
Let
$$
A(\eta)
=
\begin{cases}
1& \text{ if } \eta\geq 1,
\\
1/\eta^2& \text{ if } \eta< 1.
\end{cases}
$$
We prove approximate observability under the following geometric condition
\begin{equation}\label{eq:cond}
    {(\rmin{1})}^2 A(\eta_2)\neq {(\rmin{2})}^2 A(\eta_1).
\end{equation}

\begin{theorem}\label{th:main}
Assume \eqref{eq:cond} holds.
Assume $f\in C^2([0, \tmax]; \R)$ has a finite number of zeros
and $h(r)=1/r^m$ for some $m\in\mathbb{N}$.
Let $\psi_{i, 0}\in H^2(\rinf{i}, \rsup{i})$ and
denote by $\psi_i$ for $i=1,2$ the corresponding
solution of \eqref{syst}, satisfying the condition \eqref{eq:rmax}.
If
$\opc( \psi_1(t, \cdot),  \psi_2(t, \cdot))= 0$ for all $t\in[0, \tmax]$, then 
 $(\psi_1,\psi_2)=(0,0)$.
\end{theorem}

\begin{remark}
This statement generalizes the result of \cite{brivadis:hal-03053999},
which was limited to the case where $1 = \eta_1<\eta_2$, $\rmin{1}(0) = \rmin{2}(0)$, $\rmax{1}(0) = \rmax{2}(0)$, $g_1>0$, $g_2>0$ and $h(r) = 1$.
This case, not included in the statement of Theorem~\ref{th:main},
can be recovered thanks to technical comments found in Remark~\ref{rem:main}.
\end{remark}

We prove Theorem~\ref{th:main} in two steps. First the observability condition is translated into a sequential equality. Then we prove that this equality between sequences is actually asymptotically incompatible.

\smallskip
\noindent\textbf{Step 1}: From observability to sequence comparisons.

From \eqref{eq:ker}, we can derive 
(from the power series expansion of $\ell\mapsto\sqrt{1-\ell^2}$)
a power series expansion at $0$ (with infinite convergence radius) of $\K_i$,
$$
\K_i(\psi_i)(\ell)=\sum_{n=1}^\infty  a_n(\ee_i)b_n
\F_i^{2n}(\psi_i) \ell^{2n}
$$
with
$
a_n(\ee)=
\int_{\phi=0}^{2\pi}\int_{\theta=0}^\pi
\alpha_\ee^{n}(\phi,\theta) \frac{\sin \theta}{4\pi} \diff \theta\diff \phi$,
$b_n = \frac{1}{(n!)^2(1-2n)4^{2n}}$ and 
$$
\F_i^n(\psi_i)
=
\int_{\rmin{i}}^{\rmax{i}}
\frac{\psi_i(r)}{r^{n}}\diff r. 
$$
In proving the approximate observability, we may as well assume 
\begin{equation}\label{E:injectivity}
\K_1(\psi_1(t))=\K_2(\psi_2(t)),\qquad t\in[0,\tmax].
\end{equation}
This is reduced to power series expansion comparisons. Term-wise, we have
\begin{equation}\label{E:id}
  a_n(\eta_1) \F_1^{2n}(\psi_1)=a_n(\eta_2) \F_2^{2n}(\psi_2).
\end{equation}

If $f(t)\neq 0$ ($f$ is continuous and vanishes finitely many times), we differentiate \eqref{E:injectivity} with respect to time to obtain
$$
\frac{1}{f(t)}\partial_t
\left(
    \frac{1}{f(t)}\partial_t
    \left(
        \K_1(\psi_1)
    \right)
\right)
=
\frac{1}{f(t)}\partial_t
\left(
    \frac{1}{f(t)}\partial_t
    \left(
        \K_2(\psi_2)
    \right)
\right).
$$
Since
$
\frac{1}{f(t)}\partial_t
\left(
    \frac{1}{f(t)}\partial_t
        \psi_i
\right)=
g_i^2h(r)\partial_r(h(r)\partial_r\psi_i)
$
 (for $i=1,2$),
we obtain
\begin{equation}\label{E:id_diff}
  g_1^2 a_n(\eta_1) 
  \F_1^{2n}(h\partial_r(h\partial_r\psi_1))
  =
  g_2^2 a_n(\eta_2) 
  \F_2^{2n}(h\partial_r(h\partial_r\psi_2)),
\end{equation}
where, again,
$$
\F_i^{2n}(h\partial_r(h\partial_r\psi_i))
=
\int_{\rmin{i}}^{\rmax{i}}
\frac{h(r)\partial_r(h(r)\partial_r\psi_i(r))}{r^{2n}}\diff r.
$$
\textit{Notation:} To unburden the notations, we will denote $r_i=\rmin{i}$ for the remaining of the section.

With $h(r)=1/ r^m$ and by integration by parts,
\begin{multline*}
\begin{aligned}
\F_i^{2n}(h\partial_r(h\partial_r\psi_i))
=&-\frac{\partial_r\psi_i(r_i)}{r_i^{2n+2m}}
\\
&+
(2n+m)
\F_i^{2n+2m+1}(\partial_r\psi_i)
\\
=&-\frac{\partial_r\psi_i(r_i)}{r_i^{2n+2m}}
-(2n+m)\frac{\psi_i(r_i)}{r_i^{2n+2m+1}}
\end{aligned}
\\
+(2n+m)(2n+2m+1)
\F_i^{2n+2m+2}(\psi_i)
\end{multline*}
That is, changing the variable $n$ to $ n-m$,
\begin{multline}\label{E:IPP}
    \F^{2(n-m)}_i(h\partial_r(h\partial_r\psi_i))
    =
    -\frac{\partial_r\psi_i(r_i)}{r_i^{2n}}
    -(2n-m)\frac{\psi_i(r_i)}{r_i^{2n+1}}
    \\+
    (2n-m)(2n+1)
    \F^{2n+2}_i(\psi_i).
\end{multline}
Now we express \eqref{E:id_diff} in terms of \eqref{E:IPP}:
\begin{multline*}
    \frac{g_2^2}{g_1^2}\frac{a_{n-m}(\eta_2)}{a_{n-m}(\eta_1)}\bigg[-\frac{\partial_r\psi_2(r_2)}{r_2^{2n}}
    -(2n-m)\frac{\psi_2(r_2)}{r_2^{2n+1}}
    \\+
    (2n-m)(2n+1)
    \F^{2n+2}_1(\psi_2)\bigg]
    \qquad \\ \qquad
    =-\frac{\partial_r\psi_1(r_1)}{r_1^{2n}}
    -(2n-m)\frac{\psi_1(r_1)}{r_1^{2n+1}}
    \\+
    (2n-m)(2n+1)
    \F^{2n+2}_1(\psi_1).
\end{multline*}
Finally, we switch $\F^{2n+2}_2(\psi_2)$ for $\F^{2n+2}_1(\psi_1)$ using \eqref{E:id}, leading to the sequence equality
\begin{equation}\label{E:final_eq}
    U_n=V_n, \qquad \forall  n\geq m+1, n\in \mathbb{N},
\end{equation}
where we have set 
$$
U_n=\frac{g_2^2}{g_1^2}\frac{a_{n-m}(\eta_2)}{a_{n-m}(\eta_1)}\bigg(-\frac{\partial_r\psi_2(r_2)}{r_2^{2n}}
    -(2n-m)\frac{\psi_2(r_2)}{r_2^{2n+1}}\bigg)
$$
\begin{multline*}
    V_n=    -\frac{\partial_r\psi_1(r_1)}{r_1^{2n}}
    -(2n-m)\Bigg[\frac{\psi_1(r_1)}{r_1^{2n+1}} +
    (2n+1)\times
    \\
    \Big(1-\frac{g_2^2}{g_1^2}\frac{a_{n-m}(\eta_2)}{a_{n+1}(\eta_2)}\frac{a_{n+1}(\eta_1)}{a_{n-m}(\eta_1)}\Big)
    \F^{2n+2}_1(\psi_1)\Bigg].
\end{multline*}

\smallskip
\noindent
\textbf{Step 2:} Asymptotical identities.

To prove the approximate observability result, we prove that the asymptotics of both sides of \eqref{E:final_eq} are incompatible, imposing $(\psi_1,\psi_2)= (0,0)$.
To achieve the comparison, we need three identities. First, if $\psi_i(r_i)\neq 0$, then
\begin{equation}\label{E:id1}
    \F^{2n}_i(\psi_i)\sim \frac{\psi_i(r_i)}{2nr_i^{2n-1}}.
\end{equation}
This can be obtained by comparing $\psi_i(r)$ to $\psi_i(r_{i})$ on any small interval $(r_{i},r_{i}+\delta)$ ($\delta>0$).
By integration by parts, we also get that if $\psi_i(r_i)=0$ but $\partial_r\psi_i(r_i)\neq 0$, then
\begin{equation}\label{E:id2}
\F^{2n}_i(\psi_i)\sim \frac{\partial_r \psi_i(r_i)}{4n^2r_i^{2n-2}}.
\end{equation}
Finally,
we have
$
\lim_{n\to +\infty}\frac{a_{n+1}(\eta)}{a_n(\eta)}
=
A(\eta).
$
This can be obtained thanks to the following remark. The function $(\phi,\theta)\mapsto\frac{\sin\theta}{4\pi}$ is the density of a probability measure $\mu$ on $(\phi,\theta)\in [0,2\pi]\times [0,\pi]$. 
As such, $a_{n}(\ee)=\mathbb{E} \left(\alpha_\ee^{n}\right)$ where $\mathbb{E}$ denotes the expected value with respect to $\mu$.
In that respect, 
$$
a_{n+1}(\ee)=\mathbb{E} \left(\alpha_\ee^{n+1}\right)\leq \|\alpha_\ee\|_\infty\mathbb{E} \left(\alpha_\ee^{n}\right)=\|\alpha\|_\infty a_{n}(\ee).
$$
On the other hand,
Jensen's inequality for $\mu$ yields 
$$
a_{n+1}(\ee)
=
\mathbb{E} \left(\left(\alpha_\ee^{n}\right)^{\frac{n+1}{n}}\right)
\geq 
\left(\mathbb{E} \left(\alpha_\ee^{n}\right)\right)^{1+\frac{1}{n}}
=
\left(a_{n}(\ee)\right)^{1+\frac{1}{n}}.
$$
Hence
$\left(a_{n}(\ee)\right)^{\frac{1}{n}}
\leq 
\frac{a_{n+1}(\ee)}{a_{n}(\ee)}
\leq \|\alpha_\ee\|_\infty
$.
We obtain the result by noticing that both sides converge to $\|\alpha_\ee\|_\infty=A(\eta)$.

\smallskip

\noindent
\textbf{Step 3:} Proof of Theorem~\ref{th:main}.

\begin{proof}
First, let us analyse the influence of border terms.

Assuming either $\psi_2(r_2)\neq 0$ or  $\partial_r\psi_2(r_2)\neq 0$, the quotient $U_{n+1}/U_n$, yields 
\begin{multline*}
\frac{
    \partial_r\psi_2(r_2)r_2
    +(2n+2-m) \psi_2(r_2)
    }{
    \partial_r\psi_2(r_2)r_2^3
    +(2n-m) \psi_2(r_2)r_2^2
    }\times\\
\frac{a_{n-m+1}(\eta_2)}{a_{n-m}(\eta_2)}\frac{a_{n-m}(\eta_1)}{a_{n-m+1}(\eta_1)},
\end{multline*}
which has limit $\frac{A(\eta_2)}{r_2^2A(\eta_1)}$.

For the treatment of $V_n$, we use a natural generalization of the limit quotient of $a_n$:
$$
\frac{a_{n-m}(\eta_2)}{a_{n+1}(\eta_2)}\frac{a_{n+1}(\eta_1)}{a_{n-m}(\eta_1)}
\to
\frac{A(\eta_1)^{m+1}}{A(\eta_2)^{m+1}}.
$$

If $\psi_1(r_1)\neq 0$, then we deduce from \eqref{E:id1} that $\frac{V_n r_1^{2n}}{2n}$ has limit
$
-\frac{g_2^2}{g_1^2}\frac{A(\eta_1)^{m+1}}{A(\eta_2)^{m+1}}
\psi_1(r_1)$, which is incoherent with a limit quotient of $\frac{A(\eta_2)}{r_2^2 A(\eta_1)}\neq \frac{1}{r_1^2}$ by assumption \eqref{eq:cond}.

If $\psi_i(r_i)=0$ but $\partial_r\psi_i(r_i)\neq 0$, then
we deduce from \eqref{E:id1} that $\frac{V_n r_1^{2n}}{4n ^2}$ has limit
$
-\frac{g_2^2}{g_1^2}\frac{A(\eta_1)^{m+1}}{A(\eta_2)^{m+1}}
\partial_r\psi_1(r_1)$ which is again incoherent with a limit quotient of $\frac{A(\eta_2)}{r_2^2 A(\eta_1)}\neq \frac{1}{r_1^2}$.

Hence having $\psi_2(r_2)\neq 0$ or  $\partial_r\psi_2(r_2)\neq 0$, is incoherent with having $\psi_1(r_1)\neq 0$ or  $\partial_r\psi_1(r_1)\neq 0$. Now let's assume that $\psi_2(r_2)=\partial_r\psi_2(r_2)= 0$. Then if $\psi_1(r_1)\neq 0$, the $ \frac{V_nr_1^{2n}}{2n}$ has a non-zero limit despite being constantly zero, which is excluded. The same goes if $\partial \psi_1(r_1)\neq 0$ while $\psi_1(r_1)=0$. 

The conclusion of this first step is that if there exists a pair $(\psi_1,\psi_2)$ of $C^2$ functions satisfying \eqref{E:final_eq}, they must satisfy 
$$
\psi_1(r_1)=\partial_r\psi_1(r_1)=\psi_2(r_2)=\partial_r\psi_2(r_2)=0.
$$
We are now in a suitable position to conclude focusing on interior terms.
In that case, we are left with the equality
$$
\left(1-\frac{g_2^2}{g_1^2}\frac{a_{n-m}(\eta_2)}{a_{n+1}(\eta_2)}\frac{a_{n+1}(\eta_1)}{a_{n-m}(\eta_1)}\right)
    \F^{2n+2}_1(\psi_1)=0.
$$
Naturally, $\F^{2n+2}_1(\psi_1)$ must have  infinitely many non-zero terms, otherwise $\psi_1=0$ (the family $(1/r^{2n})_{n\geq n_0}$ is total on any bounded interval in $(a,b)$, $0<a<b$, for any $n_0$). But this would imply that 
$$
\frac{a_{n-m}(\eta_2)}{a_{n+1}(\eta_2)}\frac{a_{n+1}(\eta_1)}{a_{n-m}(\eta_1)}=\frac{g_1^2}{g_2^2}.
$$
infinitely often, which is not true except if $\eta_1=\eta_2$ and $g_1=g_2$.
\end{proof}

\begin{remark}\label{rem:main}
In the case $r_1=r_2$, $\eta_1=1$ and $\eta_2>1$,
Theorem~\ref{th:wp} does not allow to answer but the approximate observability result still holds due the following observation.
In that case, $a_n(\eta_2)\to 0$ but $1>a_n(\eta_2)\geq 1/\sqrt{n}.$ Hence, in \eqref{E:final_eq}, $V_n\times r_1^{2n}$ is equivalent to an integer power in $n$, while $U_n\times r_1^{2n}$ cannot, because of the dominating term containing $a_{n-m}(\eta_2)$.
\end{remark}

\section{Observer and numerical simulations}

The observability analysis of the previous section guarantees the convergence of the state estimation by a BFN algorithm.
Recall that the goal is to estimate $\psi_{i, 0}$ from the measurement of the CLD $Q$ over $[0, \tmax]$.
The BFN algorithm consists in applying iteratively of forward and backward Luenberger observers.
After each iteration of an observer over $[0, \tmax]$, the final estimation obtained at $\tmax$ is used as the initial condition of the next observer.
This strategy has been used in various contexts in recent decades \cite{auroux2005back, auroux2008nudging, auroux2012back}.
As shown in \cite{haine2014recovering} (which extended the results of \cite{ito2011time, Ramdani} which focused on exactly observable systems), the type of convergence depends on the observability properties of the system.
These results have been extended to the non-autonomous context (which is the case here since $G_i$ is time-varying) in \cite{brivadis:hal-02529820}
and applied to a crystallization process in \cite{brivadis:hal-03053999}.

In the context of this paper, the forward and backward observers are given by:
\begin{align}
&
\left\{\begin{aligned}
&\begin{aligned}
\frac{\partial \etath_i^{2n}}{\partial t}(t, r) = &-\vit_i(t, r) \frac{\partial \etath^{2n}_i}{\partial r}(t, r)\\
& - \mu\opc^*(\opc (\etath^{2n}_1(t), \etath^{2n}_2(t))- Q(t))
\end{aligned}
\\
&\etath^{2n}_i(0, r) =
\begin{cases}
\etath^{2n-1}_i(0, r) & \text{if } n\geq 1\\
\etath_{i, 0}(r) & \text{otherwise}
\end{cases}
\end{aligned}\right.
\label{obs2}\\
&
\left\{\begin{aligned}
&\begin{aligned}
\frac{\partial \etath^{2n+1}_i}{\partial t}(t, r) = &-\vit_i(t, r) \frac{\partial \etath_i^{2n+1}}{\partial r}(t, r)\\ 
& + \mu\opc^*(\opc (\etath^{2n+1}_1(t), \etath^{2n+1}_2(t))- Q(t))
\end{aligned}
\\
&\etath^{2n+1}_i(\tmax, r) = \etath^{2n}_i(\tmax, r)
\end{aligned}\right.
\label{obs2b}
\end{align}
where $\hat{\psi}_i^{n}(t, r)$ represents the estimation of $\psi_i(t, r)$ obtained by the algorithm after $n$ iterations,
$Q(t) = \opc(\psi_1(t), \psi_2(t))$ is the CLD at time $t$,
$\mu$ is a degree of freedom, called the observer gain,
and $\opc^*$ is the adjoint of the operator $\opc$:
\fonction{\opc^*}{Y}{X_1\times X_2}{Q}{\left(r\mapsto
\int_{0}^{\lmax} \noy_i(\ell, r)Q(\ell)\diff \ell\right)_{1\leq i\leq 2}.}
Note that \eqref{obs2} is the usual infinite-dimensional Luenberger observer of \eqref{syst}, while \eqref{obs2b} is a Luenberger observer of \eqref{syst} when reversed in time.
Then, 
combining the observability analysis provided in Theorem~\ref{th:main} and the convergence result \cite[Theorem 4.2]{brivadis:hal-03053999},
we obtain the following result.
\begin{theorem}\label{th:conv}
Under the assumptions of Theorem~\ref{th:main}, for all $\mu>0$, all $t\in[0, \tmax]$ and almost all $r\in[r_0, r_1]$,
\begin{equation}
    \etath^n(t, r)
    \cvl{n\cv+\infty}\psi(t, r).
\end{equation}
\end{theorem}

We propose a numerical simulation of this algorithm.
System~\eqref{syst} and observer \eqref{obs2}-\eqref{obs2b} being transport equations, they are solved by the method of characteristics.
The characteristic equation is given by
\begin{equation}\label{carac}
    \frac{\diff \rho_i}{\diff t} = G_i(t, \rho_i(t)).
\end{equation}
Along the solutions of this ODE,
$\psi_i$ and $\hat{\psi}_i^{n}$ satisfy
\begin{align}
    &\frac{\diff}{\diff t}\psi_i(t, \rho_i(t)) = 0,
    \label{ode1}\\
    &\frac{\diff}{\diff t}\hat{\psi}^{n}_i(t, \rho_i(t)) = (-1)^{2n+1}
    \mu\opc^*(\opc (\etath^{n}_1(t), \etath^{n}_2(t))- Q(t)).
    \label{ode2}
\end{align}

We choose a spatial discretization of $[\rinf{i}, \rsup{i}]$ with space-step $\diff x$, and integrate the characteristic equation \eqref{carac} over $[0, \tmax]$ with time-step $\diff t$ with initial conditions in this spatial discretization. Then, we integrate ODEs \eqref{ode1}-\eqref{ode2} along the characteristic curves with a first order Euler method.
Concerning the operator $\opc$, integrals are computed with the rectangles methods.
We consider the set of parameters given in Table~\ref{tab:param}, and the observer gain $\mu=0.001$ (small enough to preserve stability of the numerical scheme).
Note that in this example, $G$ does not depends on $t$, but this actually does not affect the convergence properties, since it is always possible to use a time-reparametrization as it is done in the proof of Theorem~\ref{th:wp}.
Moreover, condition \eqref{eq:cond} is satisfied by the example.
\begin{table}[ht!]
    \centering
    \begin{tabular}{|c|c|c|c|}
    \hline
        $\rmin{1}=\rmin{2} =  0.1$
        &$\rmax{1}=\rmax{2} =  0.2$
        &$\tmax = 1$
        \\
        \hline
        $g_1 = 0.1$
        &$g_2 = 0.2$
        &$f(t)h(r) = 1/r^2$
        \\
        \hline
        $\ee_1 = 0.5$
        &$\ee_2 = 2$
        &$\diff x=\diff t = 0.01$
        \\
        \hline
    \end{tabular}
    \caption{Parameters of the numerical simulation.}
    \label{tab:param}
\end{table}

The observer is initialized at $\hat\psi_{1, 0}=\hat\psi_{2, 0}=0$.
The initial conditions $\psi_{1, 0}$ and $\psi_{2, 0}$ are chosen as normal distributions centered at $r = 0.05$ and $r=0.15$, respectively.
Roughly speaking, crystals of shape $\eta_1$ will appear during $[0, \tmax]$, but are not in the reactor at $t=0$, while crystals of shape $\eta_2$ are in the reactor at $t=0$ but disappear through the process.
The result of the simulation is presented in Figure~\ref{fig:simu} (numerical implementation can be found in repository \cite{git}).
After only $10$ iterations,
the locus of the maximum of the two PSDs is already well estimated.
In practice, this is the main information to be estimated.
After $1000$ iterations, the estimations are much more accurate.
Still, a peak at $r=0.15$ remains on $\hat{\psi}_{1, 0}$ while it is not in $\psi_{1, 0}$. This peak is due to the important contribution of $\psi_{2, 0}$ in the CLD at $r=0.15$. However, its amplitude decreases as the number of iterations increases, and eventually vanishes according to Theorem~\ref{th:conv}.
\begin{figure}[ht!]
    \centering

    \includegraphics{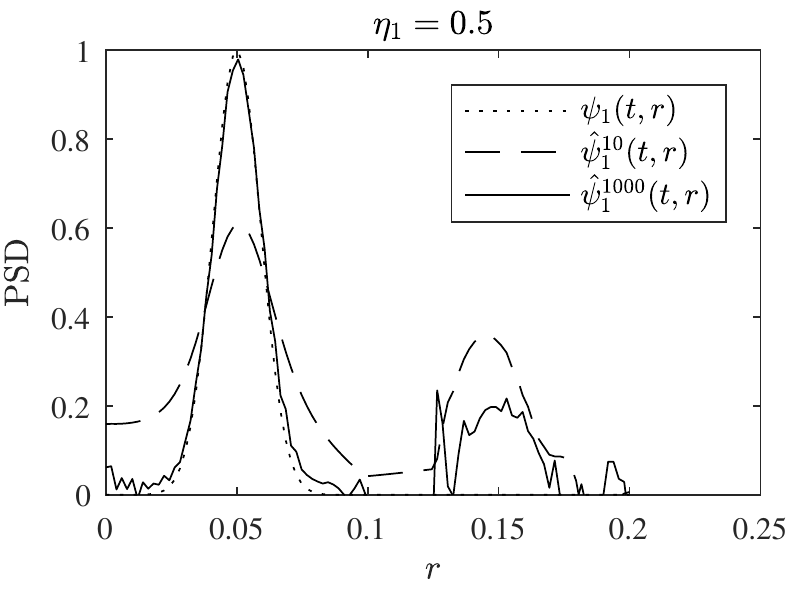}
    
    \includegraphics{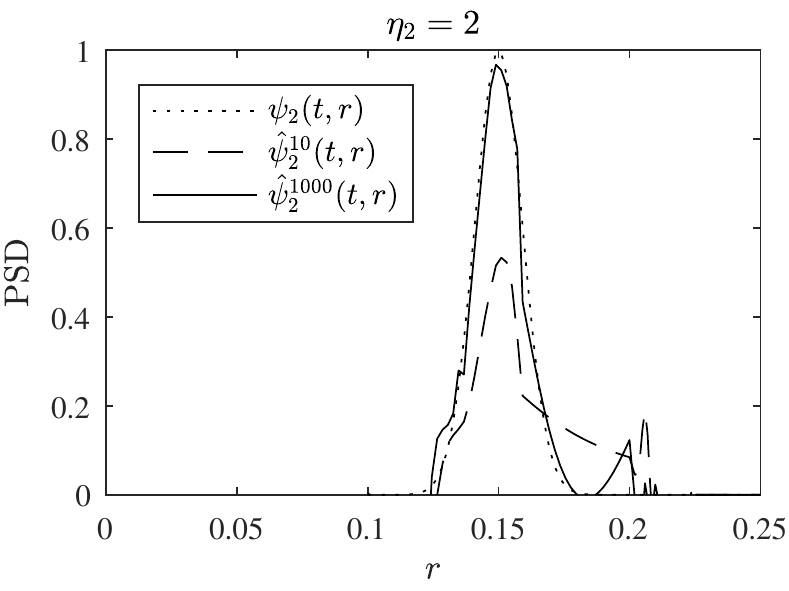}

    \caption{PSDs $\psi_1$ and $\psi_2$ at time $t=0$
    and their estimations $\hat{\psi}^{2n}_{1}(0)$ and $\hat{\psi}^{2n}_{2}(0)$ obtained after $20$ and $100$ iterations of the BFN algorithm.}
    \label{fig:simu}
\end{figure}

\section{Conclusion}

In this paper,
we propose an observability analysis of a crystallization process.
We prove, under a geometric condition,
that two PSDs
of spheroid crystals of different shapes
are fully determined by their common CLD along the process.
Hence, the BFN algorithm is able to reconstruct the PSDs from the measurement of the CLD over a finite time interval,
by using iterations of forward and backward infinite-dimensional Luenberger observers. We provide a numerical simulation of the algorithm which suggest that possible applications of this method to experimental data could benefit from this theoretical study.

\bibliographystyle{plain}
\bibliography{references}

\end{document}